\documentclass{amsart}

\usepackage[T1]{fontenc}
\usepackage{lmodern}
\usepackage{amsmath}

\usepackage[final,factor=1100,stretch=10,shrink=10]{microtype}

\SetProtrusion{encoding={*},family={bch},series={*},size={6,7}}
              {1={ ,750},2={ ,500},3={ ,500},4={ ,500},5={ ,500},
               6={ ,500},7={ ,600},8={ ,500},9={ ,500},0={ ,500}}

\SetExtraKerning[unit=space]
    {encoding={*}, family={bch}, series={*}, size={footnotesize,small,normalsize}}
    {\textendash={400,400}, % en-dash, add more space around it
     "28={ ,150}, % left bracket, add space from right
     "29={150, }, % right bracket, add space from left
     \textquotedblleft={ ,150}, % left quotation mark, space from right
     \textquotedblright={150, }} % right quotation mark, space from left

\usepackage{booktabs}

\usepackage{amsthm, enumerate}

\swapnumbers

\theoremstyle{definition}
\newtheorem{nul}{}[section]
\newtheorem{dfn}[nul]{Definition}

\newtheorem{cnstr}[nul]{Construction}
\newtheorem{ntn}[nul]{Notation}
\newtheorem{exm}[nul]{Example}

\newtheorem*{dfn*}{Definition}
\newtheorem*{axm*}{Axiom}
\newtheorem*{ntn*}{Notation}
\newtheorem*{exm*}{Example}
\newtheorem*{exr*}{Exercise}
\newtheorem*{int*}{Intuition}
\newtheorem*{qst*}{Question}

\theoremstyle{plain}

\newtheorem{thm}[nul]{Theorem}
\newtheorem{prp}[nul]{Proposition}
\newtheorem{lem}[nul]{Lemma}

\newtheorem{cor}[nul]{Corollary}
\newtheorem*{thm*}{Theorem}
\newtheorem*{mainthm*}{Universal Additivity Theorem}
\newtheorem*{structurethm*}{Structure Theorem}
\newtheorem*{prp*}{Proposition}
\newtheorem*{cor*}{Corollary}
\newtheorem*{lem*}{Lemma}
\newtheorem*{cnj*}{Conjecture}

\numberwithin{equation}{nul}

\usepackage{enumerate}

\DeclareMathOperator{\colim}{colim}
\DeclareMathOperator{\Fun}{Fun}
\DeclareMathOperator{\id}{id}
\DeclareMathOperator{\Mor}{Mor}

\newcommand{\XX}{\mathbf{X}}
\newcommand{\YY}{\mathbf{Y}}

\DeclareMathOperator{\Map}{Map}

\newcommand{\Cat}{\mathbf{Cat}}
\newcommand{\Top}{\mathbf{Top}}

\newcommand{\cart}{\mathrm{cart}}
\newcommand{\cocart}{\mathrm{cocart}}
\newcommand{\eff}{\mathit{eff}}
\newcommand{\op}{\mathit{op}}

\newcommand{\coloneq}{\mathrel{\mathop:}=}

\makeatletter 
\def\revddots{\mathinner{\mkern1mu\raise\p@ 
\vbox{\kern7\p@\hbox{.}}\mkern2mu 
\raise4\p@\hbox{.}\mkern2mu\raise7\p@\hbox{.}\mkern1mu}} 
\makeatother

\newcommand{\xdots}{\ddots\!\!\!\!\!\!\!\!\revddots}

\newcommand{\D}{\Delta}
\newcommand{\del}{\partial}

\newcommand{\mc}{\mathcal}
\newcommand{\sk}{\text{sk}}
\newcommand{\tsk}{\widetilde{\sk}}
\usepackage{tikz}
\usetikzlibrary{matrix,arrows,decorations}

\pgfdeclaredecoration{single line}{initial}{\state{initial}[width=\pgfdecoratedpathlength-1sp]{\pgfpathmoveto{\pgfpointorigin}}\state{final}{\pgfpathlineto{\pgfpointorigin}}} 

\newcommand{\fromto}[2]{{#1}\ \tikz[baseline]\draw[>=stealth,->](0,0.5ex)--(0.5,0.5ex);\ {#2}}
\newcommand{\into}[2]{{#1}\ \tikz[baseline]\draw[>=stealth,right hook->](0,0.5ex)--(0.5,0.5ex);\ {#2}}

\newcommand{\cofto}[2]{{#1}\ \tikz[baseline]\draw[>=stealth,>->](0,0.5ex)--(0.5,0.5ex);\ {#2}}
\newcommand{\fibto}[2]{{#1}\ \tikz[baseline]\draw[>=stealth,->>](0,0.5ex)--(0.5,0.5ex);\ {#2}}
\newcommand{\equivto}[2]{{#1}\ \tikz[baseline]\draw[>=stealth,->,font=\scriptsize,inner sep=0.5pt](0,0.5ex)--node[above]{$\sim$}(0.5,0.5ex);\ {#2}}
\newcommand{\goesto}[2]{{#1}\ \tikz[baseline]\draw[|->](0,0.5ex)--(0.5,0.5ex);\ {#2}}

\renewcommand{\to}{\ \tikz[baseline]\draw[>=stealth,->](0,0.5ex)--(0.5,0.5ex);\ }
\newcommand{\ot}{\ \tikz[baseline]\draw[>=stealth,<-](0,0.5ex)--(0.5,0.5ex);\ }

\title{Dualizing cartesian and cocartesian fibrations}
\author{Clark Barwick}
\author{Saul Glasman}
\author{Denis Nardin}

\begin{document}
\maketitle
\tableofcontents

Anyone who has worked seriously with quasicategories has had to spend some quality time with cartesian and cocartesian fibrations. (For a crash course in the basic definitions and constructions, see Appendix \ref{sect:recall}; for an in-depth study, see \cite[\S 2.4.2]{HTT}.) The purpose of (co)cartesian fibrations is to finesse the various homotopy coherence issues that naturally arise when one wishes to speak of functors valued in the quasicategory \(\Cat_{\infty}\) of quasicategories. A cartesian fibration \(p:\fromto{X}{S}\) is ``essentially the same thing'' as a functor \(\XX:\fromto{S^{\op}}{\Cat_{\infty}}\), and a cocartesian fibration \(q:\fromto{Y}{T}\) is ``essentially the same thing'' as a functor \(\YY:\fromto{T}{\Cat_{\infty}}\). We say that the (co)cartesian fibration $p$ or $q$ is \textbf{\emph{classified by}} $\XX$ or $\YY$ (\ref{rec:straighten}).

It has therefore been a continual source of irritation to many of us who work with quasicategories that, given a cartesian fibration $p:\fromto{X}{S}$, there has been no \emph{explicit} way to construct a cocartesian fibration $p^{\vee}:\fromto{X^{\vee}}{S^{\op}}$ that is classified by the same functor $\fromto{S^{\op}}{\Cat_{\infty}}$. Many constructions require as input exactly one of these two, and if one has become sidled with the wrong one, then in the absence of an explicit construction, one is forced to extrude the desired fibration through tortuous expressions such as ``the cocartesian fibration $p^{\vee}$ classified by the functor by which the cartesian fibration $p$ is classified.'' We know of course that such a thing exists, but we have little hope of \emph{using} it if we don't have access to a model that lets us precisely identify an $n$-simplex of $X^{\vee}$ in terms of $p$.

In this technical note, we put an end to this maddening state of affairs: we proffer a very explicit construction of the \emph{dual cocartesian fibration} $p^{\vee}$ of a cartesian fibration $p$, and we show they are classified by the same functor to $\Cat_{\infty}$. Amusingly, the construction of the dual itself is quite simple; however, proving that it works as advertised (and for that matter, even proving that $p^{\vee}$ is a cocartesian fibration) is a nontrivial matter. The main technical tool we use is the technology of \emph{effective Burnside $\infty$-categories} and the \emph{unfurling construction} of the first author \cite{mack1}.

In the first section, we will give an informal but very concrete description of the dual, and we will state the main theorem, Th. \ref{thm:main}. Some users of this technology will be happy to stop reading right there. For those who press on, in \S \ref{sect:twarr}, we briefly recall the definition of the twisted arrow category, which plays a significant role in the construction. In \S \ref{sect:dualdfn}, we give a precise definition of the dual of a cartesian fibration, and we prove that it is a cocartesian fibration. In particular, we can say \emph{exactly} what the $n$-simplices of $X^{\vee}$ are (\ref{nul:simplicesofdual}). In \S \ref{sect:veevee}, we prove Pr. \ref{prp:doubledual}, which asserts that the double dual is homotopic to the identity, and we use this to prove the main theorem, Th. \ref{thm:main}. Finally, in \S \ref{reltwarr}, we construct a relative version of the twisted arrow $\infty$-category for a cocartesian fibration and its dual.  provides another way to witness the equivalence between the functor classifying $p$ and the functor classifying $p^{\vee}$.

%----------------------------------------------------------------------%

\section{Overview}

\begin{nul} Before we describe the construction, let us pause to note that simply taking opposites will \emph{not} address the issue of the day: if $p:\fromto{X}{S}$ is a cartesian fibration, then it is true that $p^{\op}:\fromto{X^{\op}}{S^{\op}}$ is a cocartesian fibration, but the functor $\fromto{S^{\op}}{\Cat_{\infty}}$ that classifies $p^{\op}$ is the composite of the functor $\XX:\fromto{S^{\op}}{\Cat_{\infty}}$ that classifies $p$ with the involution
\begin{equation*}
\op:\fromto{\Cat_{\infty}}{\Cat_{\infty}}
\end{equation*}
that carries a quasicategory to its opposite.

This discussion does, however, permit us to rephrase the problem in an enlightening way: the morphism $(p^{\vee})^{\op}:\fromto{(X^{\vee})^{\op}}{S}$ must be another cartesian fibration that is classified by the composite of the functor that classifies $p$ with the involution $\op$. The dual cocartesian fibration to $(p^{\vee})^{\op}$ should be equivalent to $p^{\op}$, so that we have a duality formula
\begin{equation*}
((p^{\vee})^{\op})^{\vee}\simeq p^{\op}.
\end{equation*}
In particular, it will be sensible to define the dual $q^{\vee}$ of a \emph{co}cartesian fibration $q:\fromto{Y}{T}$ as $((q^{\op})^{\vee})^{\op}$, so that $p^{\vee\vee}\simeq p$. We thus summarize:
\begin{center}
\smallskip
\begin{tabular}{ccc}
\toprule
The cartesian fibration & and the cocartesian fibration & are each classified by\\
\midrule
$p:\fromto{X}{S}$ & $p^{\vee}:\fromto{X^{\vee}}{S^{\op}}$ & $\XX:\fromto{S^{\op}}{\Cat_{\infty}}$;\\
$(p^{\vee})^{\op}:\fromto{(X^{\vee})^{\op}}{S}$ & $p^{\op}:\fromto{X^{\op}}{S^{\op}}$ & $\op\circ\XX:\fromto{S^{\op}}{\Cat_{\infty}}$;\\
$q^{\vee}:\fromto{Y^{\vee}}{T^{\op}}$ & $q:\fromto{Y}{T}$ & $\YY:\fromto{T}{\Cat_{\infty}}$;\\
$q^{\op}:\fromto{Y^{\op}}{T^{\op}}$ & $(q^{\vee})^{\op}:\fromto{(Y^{\vee})^{\op}}{T}$ & $\op\circ\YY:\fromto{T}{\Cat_{\infty}}$.\\
\bottomrule
\end{tabular}
\smallskip
\end{center}
\end{nul}

\begin{nul} We can describe our construction very efficiently if we give ourselves the luxury of temporarily skipping some details. For any quasicategory $S$ and any cartesian fibration $p:\fromto{X}{S}$, we will define $X^{\vee}$ as a quasicategory whose objects are those of $X$ and whose morphisms $\fromto{x}{y}$ are diagrams
\begin{equation}\label{eq:mapfromxtoyinXdual}
\begin{tikzpicture}[baseline]
\matrix(m)[matrix of math nodes, 
row sep=2ex, column sep=3ex, 
text height=1.5ex, text depth=0.25ex] 
{&u&\\ 
x&&y\\}; 
\path[>=stealth,->,inner sep=0.9pt,font=\scriptsize] 
(m-1-2) edge node[above left]{$f$} (m-2-1) 
edge node[above right]{$g$} (m-2-3); 
\end{tikzpicture}
\end{equation}
of $X$ in which $f$ is a $p$-cartesian edge, and $p(g)$ is a degenerate edge of $S$. Composition of morphisms in $X^{\vee}$ will be given by forming a pullback:
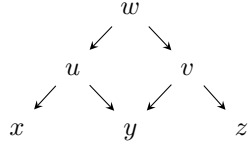
\begin{equation*}
\begin{tikzpicture} 
\matrix(m)[matrix of math nodes, 
row sep=2ex, column sep=2ex, 
text height=1.5ex, text depth=0.25ex] 
{&&w&&\\
&u&&v&\\ 
x&&y&&z\\}; 
\path[>=stealth,->,font=\scriptsize] 
(m-1-3) edge (m-2-2) 
edge (m-2-4)
(m-2-2) edge (m-3-1) 
edge (m-3-3)
(m-2-4) edge (m-3-3) 
edge (m-3-5); 
\end{tikzpicture}
\end{equation*}

The $n$-simplices for $n\geq3$ are described completely in \ref{nul:simplicesofdual}. One now has to explain why this defines a quasicategory, but it does indeed (Df. \ref{dfn:dualdefn}), and it admits a natural functor to $S^{\op}$ that carries an object $x$ to $p(x)$ and a morphism as in \eqref{eq:mapfromxtoyinXdual} to the edge $p(f):\fromto{p(x)}{p(u)=p(y)}$ in $S^{\op}$. This is our functor $p^{\vee}:\fromto{X^{\vee}}{S^{\op}}$, and we have good news.
\end{nul}

\begin{prp} If $p:\fromto{X}{S}$ is a cartesian fibration, then $p^{\vee}:\fromto{X^{\vee}}{S^{\op}}$ is a cocartesian fibration, and a morphism as in \eqref{eq:mapfromxtoyinXdual} is $p^{\vee}$-cocartesian just in case $g$ is an equivalence.
\end{prp}
\noindent This much will actually follow trivially from the fundamental unfurling lemmas of the first author \cite[Lm. 11.4 and Lm. 11.5]{mack1}, but the duality statement we're after is more than just the construction of this cocartesian fibration.

If one inspects the fiber of $p^{\vee}$ over a vertex $s\in S^{\op}$, one finds that it is the quasicategory whose objects are objects of $X_s\coloneq p^{-1}(s)$, and whose morphisms $\fromto{x}{y}$ are diagrams \eqref{eq:mapfromxtoyinXdual} of $X_s$ in which $f$ is an equivalence. This is visibly equivalent to $X_s$ itself. Furthemore, we will prove that this equivalence is functorial:
\begin{prp} The functor $\fromto{S^{\op}}{\Cat_{\infty}}$ that classifies a cartesian fibration $p$ is equivalent to the functor $\fromto{S^{\op}}{\Cat_{\infty}}$ that classifies its dual $p^{\vee}$.
\end{prp}

\noindent Equivalently, we have the following.

\begin{prp} If $\XX:\fromto{S^{\op}}{\Cat_{\infty}}$ classifies $p$, then $\op\circ\XX:\fromto{S^{\op}}{\Cat_{\infty}}$ classifies $(p^{\vee})^{\op}$.
\end{prp}

\begin{nul} We will define the dual of a \emph{co}cartesian fibration $q:\fromto{Y}{T}$ over a quasicategory $T$ as suggested above:
\begin{equation*}
q^{\vee}\coloneq((q^{\op})^{\vee})^{\op}.
\end{equation*}
In other words, $Y^{\vee}$ will be the quasicategory whose objects are those of $Y$ and whose morphisms $\fromto{x}{y}$ are diagrams
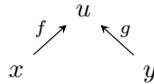
\begin{equation*}
\begin{tikzpicture} 
\matrix(m)[matrix of math nodes, 
row sep=2ex, column sep=3ex, 
text height=1.5ex, text depth=0.25ex] 
{&u&\\ 
x&&y\\}; 
\path[>=stealth,<-,inner sep=0.9pt,font=\scriptsize] 
(m-1-2) edge node[above left]{$f$} (m-2-1) 
edge node[above right]{$g$} (m-2-3); 
\end{tikzpicture}
\end{equation*}
of $Y$ in which $q(f)$ is a degenerate edge of $T$, and $g$ is $q$-cocartesian. Composition of morphisms in $Y^{\vee}$ will be given by forming a pushout:
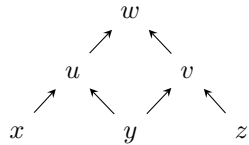
\begin{equation*}
\begin{tikzpicture} 
\matrix(m)[matrix of math nodes, 
row sep=2ex, column sep=2ex, 
text height=1.5ex, text depth=0.25ex] 
{&&w&&\\
&u&&v&\\ 
x&&y&&z\\}; 
\path[>=stealth,<-,font=\scriptsize] 
(m-1-3) edge (m-2-2) 
edge (m-2-4)
(m-2-2) edge (m-3-1) 
edge (m-3-3)
(m-2-4) edge (m-3-3) 
edge (m-3-5); 
\end{tikzpicture}
\end{equation*}
The three propositions above will immediately dualize. 
\end{nul}

In summary, the objects of $X^{\vee}$ and $(X^{\vee})^{\op}=(X^{\op})^{\vee}$ are simply the objects of $X$, and the objects of $Y^{\vee}$ and $(Y^{\vee})^{\op}=(Y^{\op})^{\vee}$ are simply the objects of $Y$. A morphism $\eta:\fromto{x}{y}$ in each of these $\infty$-categories is then as follows:
\begin{center}
\smallskip
\begin{tabular}{ccccc}
\toprule
In & $\eta$ is a diagram & of & in which $f$ & and $g$\\
\midrule
$X^{\vee}$ & $\begin{tikzpicture}[baseline]
\matrix(m)[matrix of math nodes, 
row sep=2ex, column sep=3ex, 
text height=1.5ex, text depth=0.25ex] 
{&u&\\ 
x&&y\\}; 
\path[>=stealth,->,inner sep=0.9pt,font=\scriptsize] 
(m-1-2) edge node[above left]{$f$} (m-2-1) 
edge node[above right]{$g$} (m-2-3); 
\end{tikzpicture}$ & $X$ & is $p$-cartesian, & lies over an identity;\\
$(X^{\vee})^{\op}$ & $\begin{tikzpicture}[baseline]
\matrix(m)[matrix of math nodes, 
row sep=2ex, column sep=3ex, 
text height=1.5ex, text depth=0.25ex] 
{&u&\\ 
x&&y\\}; 
\path[>=stealth,->,inner sep=0.9pt,font=\scriptsize] 
(m-1-2) edge node[above left]{$f$} (m-2-1) 
edge node[above right]{$g$} (m-2-3); 
\end{tikzpicture}$ & $X$ & lies over an identity, & is $p$-cartesian;\\
$Y^{\vee}$ & $\begin{tikzpicture}[baseline]
\matrix(m)[matrix of math nodes, 
row sep=2ex, column sep=3ex, 
text height=1.5ex, text depth=0.25ex] 
{&u&\\ 
x&&y\\}; 
\path[>=stealth,<-,inner sep=0.9pt,font=\scriptsize] 
(m-1-2) edge node[above left]{$f$} (m-2-1) 
edge node[above right]{$g$} (m-2-3); 
\end{tikzpicture}$ & $Y$ & lies over an identity, & is $q$-cocartesian;\\
$(Y^{\vee})^{\op}$ & $\begin{tikzpicture}[baseline]
\matrix(m)[matrix of math nodes, 
row sep=2ex, column sep=3ex, 
text height=1.5ex, text depth=0.25ex] 
{&u&\\ 
x&&y\\}; 
\path[>=stealth,<-,inner sep=0.9pt,font=\scriptsize] 
(m-1-2) edge node[above left]{$f$} (m-2-1) 
edge node[above right]{$g$} (m-2-3); 
\end{tikzpicture}$ & $Y$ & is $q$-cocartesian, & lies over an identity.\\
\bottomrule
\end{tabular}
\smallskip
\end{center}

The propositions above are all subsumed in the following statement of our main theorem, which employs some of the notation of \ref{rec:straighten}.

\begin{thm}\label{thm:main} The assignments $\goesto{p}{p^{\vee}}$ and $\goesto{q}{q^{\vee}}$ define homotopy inverse equivalences of $\infty$-categories
\begin{equation*}
(-)^{\vee}:\Cat_{\infty/S}^{\cart}\ \tikz[baseline]\draw[>=stealth,<->,font=\scriptsize,inner sep=0.5pt](0,0.5ex)--node[above]{$\sim$}(0.75,0.5ex);\ \Cat_{\infty/S^{\op}}^{\cocart}:(-)^{\vee}
\end{equation*}
of cartesian fibrations over $S$ and cocartesian fibrations over $S^{\op}$. These equivalences are compatible with the straightening/unstraightening equivalences $s$ in the sense that the diagram of equivalences
\begin{equation*}
\begin{tikzpicture} 
\matrix(m)[matrix of math nodes, 
row sep=4ex, column sep=2ex, 
text height=1.5ex, text depth=0.25ex] 
{\Cat_{\infty/S}^{\cart} &&\Cat_{\infty/S^{\op}}^{\cocart}\\ 
&\Fun(S^{\op},\Cat_{\infty})&\\[2ex]
&\Fun(S^{\op},\Cat_{\infty})&\\
\Cat_{\infty/S^{\op}}^{\cocart} &&\Cat_{\infty/S}^{\cart}\\}; 
\path[>=stealth,<->,inner sep=1pt,font=\scriptsize] 
(m-1-1) edge node[above]{$(-)^{\vee}$} (m-1-3) 
edge node[below left]{$s$} (m-2-2)
edge node[left]{$\op$} (m-4-1)
(m-1-3) edge node[below right]{$s$} (m-2-2)
edge node[right]{$\op$} (m-4-3)
(m-2-2) edge node[right]{$\op\circ-$} (m-3-2)
(m-4-1) edge node[below]{$(-)^{\vee}$} (m-4-3)
edge node[above left]{$s$} (m-3-2)
(m-4-3) edge node[above right]{$s$} (m-3-2); 
\end{tikzpicture}
\end{equation*}
commutes up to a (canonical) homotopy.
\end{thm}

%----------------------------------------------------------------------%

\section{Twisted arrow $\infty$-categories}\label{sect:twarr}

\begin{dfn} If $X$ is an $\infty$-category (i.e., a quasicategory), then the \emph{twisted arrow $\infty$-category} $\widetilde{\mathcal{O}}(X)$ is the simplicial set given by the formula
\begin{equation*}
\widetilde{\mathcal{O}}(X)_n=\Mor(\Delta^{n,\op}\star\Delta^n,X)\cong X_{2n+1}.
\end{equation*}
The two inclusions
\begin{equation*}
  \into{\Delta^{n,\op}}{\Delta^{n,\op}\star\Delta^n}\textrm{\quad and\quad}\into{\Delta^{n}}{\Delta^{n,\op}\star\Delta^n}
\end{equation*}
give rise to a map of simplicial sets
\begin{equation*}
\fromto{\widetilde{\mathcal{O}}(X)}{X^{\op}\times X}.
\end{equation*}
\end{dfn}

\begin{nul}\label{nul:twarrofnerveisverve} The vertices of $\widetilde{\mathcal{O}}(X)$ are edges of $X$; an edge of $\widetilde{\mathcal{O}}(X)$ from $\fromto{u}{v}$ to $\fromto{x}{y}$ can be viewed as a commutative diagram (up to chosen homotopy)
\begin{equation*}
\begin{tikzpicture} 
\matrix(m)[matrix of math nodes, 
row sep=4ex, column sep=4ex, 
text height=1.5ex, text depth=0.25ex] 
{u&x\\ 
v&y\\}; 
\path[>=stealth,->,font=\scriptsize] 
(m-1-2) edge (m-1-1) 
edge (m-2-2)
(m-1-1) edge (m-2-1)
(m-2-1) edge (m-2-2); 
\end{tikzpicture}
\end{equation*}
When $X$ is the nerve of an ordinary category $C$, $\widetilde{\mathcal{O}}(X)$ is isomorphic to the nerve of the twisted arrow category of $C$ in the sense of \cite{MR705421}. When $X$ is an $\infty$-category, our terminology is justified by the following.
\end{nul}

\begin{prp}[Lurie, \protect{\cite[Pr. 4.2.3]{DAGX}}]\label{prp:twarrisinfincat} If $X$ is an $\infty$-category, then the functor $\fromto{\widetilde{\mathcal{O}}(X)}{X^{\op}\times X}$ is a left fibration; in particular, $\widetilde{\mathcal{O}}(X)$ is an $\infty$-category.
\end{prp}

\begin{exm} To illustrate, for any object $\mathbf{p}\in\Delta$, the $\infty$-category $\widetilde{\mathcal{O}}(\Delta^p)$ is the nerve of the category
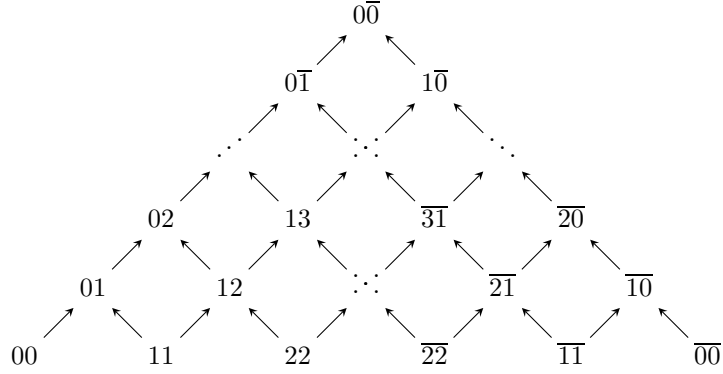
\begin{equation*}
\begin{tikzpicture} 
\matrix(m)[matrix of math nodes, 
row sep={6ex,between origins}, column sep={6ex,between origins}, 
text height=1.5ex, text depth=0.25ex] 
{&&&&&0\overline{0}&&&&&\\
&&&&0\overline{1}&&1\overline{0}&&&&\\
&&&\revddots&&\xdots&&\ddots&&&\\
&&02&&13&&\overline{3}\overline{1}&&\overline{2}\overline{0}&&\\
&01&&12&&\xdots&&\overline{2}\overline{1}&&\overline{1}\overline{0}&\\
00&&11&&22&&\overline{2}\overline{2}&&\overline{1}\overline{1}&&\overline{0}\overline{0}\\}; 
\path[>=stealth,->,font=\scriptsize] 
(m-2-5) edge (m-1-6) 
(m-2-7) edge (m-1-6)
(m-3-4) edge (m-2-5)
(m-3-6) edge (m-2-5)
(m-3-6) edge (m-2-7)
(m-3-8) edge (m-2-7)
(m-4-3) edge (m-3-4)
(m-4-5) edge (m-3-4)
(m-4-5) edge (m-3-6)
(m-4-7) edge (m-3-6)
(m-4-7) edge (m-3-8)
(m-4-9) edge (m-3-8)
(m-5-2) edge (m-4-3)
(m-5-4) edge (m-4-3)
(m-5-4) edge (m-4-5)
(m-5-6) edge (m-4-5)
(m-5-6) edge (m-4-7)
(m-5-8) edge (m-4-7)
(m-5-8) edge (m-4-9)
(m-5-10) edge (m-4-9)
(m-6-1) edge (m-5-2)
(m-6-3) edge (m-5-2)
(m-6-3) edge (m-5-4)
(m-6-5) edge (m-5-4)
(m-6-5) edge (m-5-6)
(m-6-7) edge (m-5-6)
(m-6-7) edge (m-5-8)
(m-6-9) edge (m-5-8)
(m-6-9) edge (m-5-10) 
(m-6-11) edge (m-5-10);
\end{tikzpicture}
\end{equation*}
(Here we write $\overline{n}$ for $p-n$.)
\end{exm}

In \cite[\S 4.2]{DAGX}, Lurie goes a step further and gives a characterization the left fibrations that (up to equivalence) are of the form $\fromto{\widetilde{\mathcal{O}}(X)}{X^{\op}\times X}$. We'll discuss (and use!) this result in more detail in \S \ref{reltwarr}.

%----------------------------------------------------------------------%

\section{The definition of the dual}\label{sect:dualdfn} We now give a precise definition of the dual of a cartesian fibration and, conversely, the dual of a cocartesian fibration. The definitions themselves will not depend on previous work, but the proofs that the constructions have the desired properties follow trivially from general facts about the unfurling construction of the first author \cite[Lm. 11.4 and 11.5]{mack1}.

\begin{ntn} Throughout this section, suppose $S$ and $T$ two $\infty$-categories, $p:\fromto{X}{S}$ a cartesian fibration, and $q:\fromto{Y}{T}$ a cocartesian fibration.

As in Nt. \ref{rec:leftfib}, denote by $\iota S\subset S$ the subcategory that contains all the objects and whose morphisms are equivalences. Denote by $\iota^SX\subset X$ the subcategory that contains all the objects, whose morphisms are $p$-cartesian edges.

Similarly, denote by $\iota T\subset T$ the subcategory that contains all the objects, whose morphisms are equivalences. Denote by $\iota_TY\subset Y$ the subcategory that contains all the objects and whose morphisms are $q$-cocartesian edges.
\end{ntn}

\begin{nul} It is easy to see that
\begin{equation*}
(S,\iota S,S)\textrm{\quad and\quad}(X,X\times_S\iota S,\iota^SX)
\end{equation*}
are adequate triples of $\infty$-categories in the sense of \cite[Df. 5.2]{mack1}. Dually,
\begin{equation*}
(T^{\op},\iota T^{\op},T^{\op})\textrm{\quad and\quad}(Y^{\op},Y^{\op}\times_{T^{\op}}\iota T^{\op},(\iota_{T}Y)^{\op})
\end{equation*}
are adequate triples of $\infty$-categories.

Furthermore, the cartesian fibrations $p:\fromto{X}{S}$ and $q:\fromto{Y^{{\op}}}{T^{{\op}}}$ are adequate inner fibrations over $(S,\iota S,S)$ and $(T^{\op},\iota T^{\op},T^{\op})$ (respectively) in the sense of \cite[Df. 10.3]{mack1}.
\end{nul}

\begin{dfn} For any $\infty$-category $C$ and any two subcategories $C_{\dag}\subset C$ and $C^{\dag}\subset C$ that each contain all the equivalences, we define $A^{\eff}(C,C_{\dag},C^{\dag})$ as the simplicial set whose $n$-simplices are those functors
\begin{equation*}
x:\fromto{\widetilde{\mathcal{O}}(\Delta^n)^{\op}}{C}
\end{equation*}
such that for any integers $0\leq i\leq k\leq\ell\leq j\leq n$, the square
\begin{equation*}
\begin{tikzpicture} 
\matrix(m)[matrix of math nodes, 
row sep=4ex, column sep=4ex, 
text height=1.5ex, text depth=0.25ex] 
{x_{ij}&x_{kj}\\ 
x_{i\ell}&x_{k\ell}\\}; 
\path[>=stealth,->,font=\scriptsize] 
(m-1-1) edge[>->] (m-1-2) 
edge[->>] (m-2-1) 
(m-1-2) edge[->>] (m-2-2) 
(m-2-1) edge[>->] (m-2-2); 
\end{tikzpicture}
\end{equation*}
is a pullback in which the morphisms $\cofto{x_{ij}}{x_{kj}}$ and $\cofto{x_{i\ell}}{x_{k\ell}}$ lie in $C_{\dag}$ and the morphisms $\fibto{x_{ij}}{x_{i\ell}}$ and $\fibto{x_{kj}}{x_{k\ell}}$ lie in $C^{\dag}$.

When $A^{\eff}(C,C_{\dag},C^{\dag})$ is an $\infty$-category (which is the case, for example, when $(C,C_{\dag},C^{\dag})$ is an adequate triple of $\infty$-categories in the sense of \cite[Df. 5.2]{mack1}), we call it the \textbf{\emph{effective Burnside $\infty$-category of $(C,C_{\dag},C^{\dag})$}}.

Note that the projections $\fromto{\widetilde{\mathcal{O}}(\Delta^n)^{\op}}{\Delta^n}$ and $\fromto{\widetilde{\mathcal{O}}(\Delta^n)^{\op}}{(\Delta^n)^{\op}}$ induce inclusions
\begin{equation*}
\into{C_{\dag}}{A^{\eff}(C,C_{\dag},C^{\dag})}\textrm{\quad and\quad}\into{(C^{\dag})^{\op}}{A^{\eff}(C,C_{\dag},C^{\dag})}.
\end{equation*}
\end{dfn}

Now it is easy to see that $p$ and $q$ induce morphisms of simplicial sets
\begin{equation*}
\overline{p}:\fromto{A^{\eff}(X,X\times_S\iota S,\iota^SX)}{A^{\eff}(S,\iota S,S)}
\end{equation*}
and
\begin{equation*}
\overline{q}:\fromto{A^{\eff}(Y^{\op},Y^{\op}\times_{T^{\op}}\iota T^{\op},(\iota_TY)^{\op})^{\op}}{A^{\eff}(T^{\op},\iota T^{\op},T^{\op})^{\op}},
\end{equation*}
respectively. We wish to see that $\overline{p}$ is a cocartesian fibration and that $\overline{q}$ is a cartesian fibration, but it's not even clear that they are inner fibrations.

Luckily, the fundamental unfurling lemmas \cite[Lm. 11.4 and Lm. 11.5]{mack1} of the first author address exactly this point. The basic observation is that the unfurling
\[\Upsilon(X/(S,\iota S,S))\textrm{\qquad(respectively,\quad}\Upsilon(Y^{\op}/(T^{\op},\iota T^{\op},T^{\op}))\textrm{\quad)}\]
of the adequate inner fibration $p$ (resp., $q^{\op}$) \cite[Df. 11.3]{mack1} is then the effective Burnside $\infty$-category
\[A^{\eff}(X,X\times_S\iota S,\iota^SX)\textrm{\qquad(resp.,\quad}A^{\eff}(Y^{\op},Y^{\op}\times_{T^{\op}}\iota T^{\op},(\iota_TY)^{\op})\textrm{\quad),}\]
and the functor $\Upsilon(p)$ (resp., the functor $\Upsilon(q^{\op})^{\op}$) is the functor $\overline{p}$ (resp., the functor $\overline{q}$) described above. The fundamental lemmas \cite[Lm. 11.4 and Lm. 11.5]{mack1} now immediately imply the following.
\begin{prp} The simplicial set $A^{\eff}(S,\iota S,S)$ is an $\infty$-category, and the functor $\overline{p}$ is a cocartesian fibration. Furthermore, a morphism of $A^{\eff}(X,X\times_S\iota S,\iota^SX)$ of the form
\begin{equation*}
\begin{tikzpicture} 
\matrix(m)[matrix of math nodes, 
row sep=2ex, column sep=3ex, 
text height=1.5ex, text depth=0.25ex] 
{&u&\\ 
x&&y\\}; 
\path[>=stealth,->,inner sep=0.9pt,font=\scriptsize] 
(m-1-2) edge node[above left]{$f$} (m-2-1) 
edge node[above right]{$g$} (m-2-3); 
\end{tikzpicture}
\end{equation*}
is $\overline{p}$-cocartesian just in case $g$ is an equivalence.

Dually, the simplicial set $A^{\eff}(T,T,\iota T)$ is an $\infty$-category, and the functor $\overline{q}$ is a cartesian fibration. Furthermore, a morphism of $A^{\eff}(Y^{\op},Y^{\op}\times_{T^{\op}}\iota T^{\op},(\iota_TY)^{\op})^{\op}$ of the form
\begin{equation*}
\begin{tikzpicture} 
\matrix(m)[matrix of math nodes, 
row sep=2ex, column sep=3ex, 
text height=1.5ex, text depth=0.25ex] 
{&u&\\ 
x&&y\\}; 
\path[>=stealth,<-,inner sep=0.9pt,font=\scriptsize] 
(m-1-2) edge node[above left]{$f$} (m-2-1) 
edge node[above right]{$g$} (m-2-3); 
\end{tikzpicture}
\end{equation*}
is $\overline{q}$-cocartesian just in case $f$ is an equivalence.
\end{prp}
 
\begin{dfn}\label{dfn:dualdefn} The \textbf{\emph{dual}} of $p$ is the projection
\begin{equation*}
p^{\vee}:\fromto{X^{\vee}\coloneq A^{\eff}(X,X\times_S\iota S,\iota^SX)\times_{A^{\eff}(S,\iota S,S)}S^{\op}}{S^{\op}},
\end{equation*}
which is a cocartesian fibration. Dually, the \textbf{\emph{dual}} of $q$ is the projection
\begin{equation*}
q^{\vee}:\fromto{Y^{\vee}\coloneq A^{\eff}(Y^{\op},Y^{\op}\times_{T^{\op}}\iota T^{\op},(\iota_TY)^{\op})^{\op}\times_{A^{\eff}(T^{\op},\iota T^{\op},T^{\op})^{\op}}T}{T},
\end{equation*}
which is a cartesian fibration.
\end{dfn}

\begin{nul} The formation of the dual and the formation of the opposite are by construction dual operations with respect to each other; that is, one has by definition
\begin{equation*}
(p^{\op})^{\vee}=(p^{\vee})^{\op}\textrm{\quad and\quad}(q^{\op})^{\vee}=(q^{\vee})^{\op}.
\end{equation*}
\end{nul}

\begin{nul} Observe that the inclusions
\begin{equation*}
\into{S^{\op}}{A^{\eff}(S,\iota S,S)}\textrm{\quad and\quad}\into{T}{A^{\eff}(T^{\op},\iota T^{\op},T^{\op})^{\op}}
\end{equation*}
are each equivalences. Consequently, the projections
\begin{equation*}
\fromto{X^{\vee}}{A^{\eff}(X,X\times_S\iota S,\iota^SX)}\textrm{\quad and\quad}\fromto{Y^{\vee}}{A^{\eff}(Y^{\op},Y^{\op}\times_{T^{\op}}\iota T^{\op},(\iota_TY)^{\op})^{\op}}
\end{equation*}
are equivalences as well.
\end{nul}

\begin{nul}\label{nul:simplicesofdual} Note also that the description of $X^{\vee}$ and $Y^{\vee}$ given in the introduction coincides with the one given here: an $n$-simplex of $X^{\vee}$, for instance, is a diagram
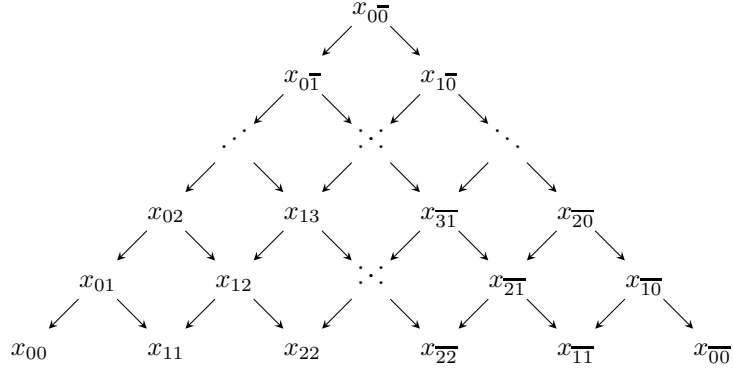
\begin{equation*}
\begin{tikzpicture} 
\matrix(m)[matrix of math nodes, 
row sep={6ex,between origins}, column sep={6ex,between origins}, 
text height=1.5ex, text depth=0.25ex] 
{&&&&&x_{0\overline{0}}&&&&&\\
&&&&x_{0\overline{1}}&&x_{1\overline{0}}&&&&\\
&&&\revddots&&\xdots&&\ddots&&&\\
&&x_{02}&&x_{13}&&x_{\overline{3}\overline{1}}&&x_{\overline{2}\overline{0}}&&\\
&x_{01}&&x_{12}&&\xdots&&x_{\overline{2}\overline{1}}&&x_{\overline{1}\overline{0}}&\\
x_{00}&&x_{11}&&x_{22}&&x_{\overline{2}\overline{2}}&&x_{\overline{1}\overline{1}}&&x_{\overline{0}\overline{0}}\\}; 
\path[>=stealth,<-,font=\scriptsize] 
(m-2-5) edge (m-1-6) 
(m-2-7) edge (m-1-6)
(m-3-4) edge (m-2-5)
(m-3-6) edge (m-2-5)
(m-3-6) edge (m-2-7)
(m-3-8) edge (m-2-7)
(m-4-3) edge (m-3-4)
(m-4-5) edge (m-3-4)
(m-4-5) edge (m-3-6)
(m-4-7) edge (m-3-6)
(m-4-7) edge (m-3-8)
(m-4-9) edge (m-3-8)
(m-5-2) edge (m-4-3)
(m-5-4) edge (m-4-3)
(m-5-4) edge (m-4-5)
(m-5-6) edge (m-4-5)
(m-5-6) edge (m-4-7)
(m-5-8) edge (m-4-7)
(m-5-8) edge (m-4-9)
(m-5-10) edge (m-4-9)
(m-6-1) edge (m-5-2)
(m-6-3) edge (m-5-2)
(m-6-3) edge (m-5-4)
(m-6-5) edge (m-5-4)
(m-6-5) edge (m-5-6)
(m-6-7) edge (m-5-6)
(m-6-7) edge (m-5-8)
(m-6-9) edge (m-5-8)
(m-6-9) edge (m-5-10) 
(m-6-11) edge (m-5-10);
\end{tikzpicture}
\end{equation*}
in which any $j$-simplex of the form $x_{0j}\to x_{1j}\to\cdots\to x_{jj}$ covers a totally degenerate simplex of $S$ (i.e., a $j$-simplex in the image of $\fromto{S_0}{S_j}$), and all the morphisms $\fromto{x_{ij}}{x_{i\ell}}$ are $p$-cartesian.

In particular, note that the fibers $(X^{\vee})_s$ are equivalent to the fibers $X_s$, and the fibers $(Y^{\vee})_t$ are equivalent to the fibers $Y_t$.
\end{nul}

%----------------------------------------------------------------------%

\section{The double dual}\label{sect:veevee}

\begin{prp}\label{prp:doubledual} Suppose $S$ and $T$ two $\infty$-categories, $p:\fromto{X}{S}$ a cartesian fibration, and $q:\fromto{Y}{T}$ a cocartesian fibration. There are natural equivalences
\begin{equation*}
p\simeq p^{\vee\vee}\textrm{\quad and\quad}q\simeq q^{\vee\vee}
\end{equation*}
of cartesian fibrations $\fromto{X}{S}$ and cocartesian fibrations $\fromto{Y}{T}$, respectively.
\end{prp}

\noindent We postpone the proof (which is quite a chore) till the end of this section. In the meantime, let us reap the rewards of our deferred labor: in the notation of \ref{rec:straighten}, we obtain the following.

\begin{cor} The formation of the dual defines an equivalence of $\infty$-categories
\begin{equation*}
(-)^{\vee}:\Cat_{\infty/S}^{\cart}\ \tikz[baseline]\draw[>=stealth,<->,font=\scriptsize,inner sep=0.5pt](0,0.5ex)--node[above]{$\sim$}(0.75,0.5ex);\ \Cat_{\infty/S^{\op}}^{\cocart}:(-)^{\vee}
\end{equation*}
\begin{proof} The only thing left to observe that $(-)^{\vee}$ is a functor from the ordinary category of cartesian (respectively, cocartesian) fibrations to the ordinary category of cocartesian (resp., cartesian) fibrations, and this functor preserves weak equivalences (since they are defined fiberwise), whence it descends to a functor of $\infty$-categories $\fromto{\Cat_{\infty/S}^{\cart}}{\Cat_{\infty/S^{\op}}^{\cocart}}$ (resp., $\fromto{\Cat_{\infty/S^{\op}}^{\cocart}}{\Cat_{\infty/S}^{\cart}}$).
\end{proof}
\end{cor}

Let's now prove the main theorem, Th. \ref{thm:main}. To do so, we must engage with some size issues.
\begin{ntn} We recall the set-theoretic technicialities and notation used in \cite[\S 1.2.15, Rk. 3.0.0.5]{HTT}. Let us choose two strongly inaccessible uncountable cardinals $\kappa<\lambda$. Denote by $\Cat_{\infty}$ (repsectively, $\Top$) $\infty$-category of $\kappa$-small $\infty$-categories (resp., of $\kappa$-small Kan complexes). Similarly, denote by $\widehat{\Cat}_{\infty}$ (resp., $\widehat{\Top}$) the $\infty$-category of $\lambda$-small $\infty$-categories (resp., of $\lambda$-small Kan complexes).

Note that $\Cat_{\infty}$ and $\Top$ are essentially $\lambda$-small and locally $\kappa$-small, whereas $\widehat{\Cat}_{\infty}$ and $\widehat{\Top}$ are only locally $\lambda$-small.
\end{ntn}

\begin{proof}[Proof of Th. \protect{\ref{thm:main}}] For any $\infty$-category $S$, consider the composite equivalence
\begin{equation*}
\Fun(S^{\op},\Cat_{\infty})\ \tikz[baseline]\draw[>=stealth,->,font=\scriptsize,inner sep=0.5pt](0,0.5ex)--node[above]{$\sim$}(0.75,0.5ex);\ \Cat_{\infty/S}^{\cart}\ \tikz[baseline]\draw[>=stealth,->,font=\scriptsize,inner sep=0.5pt](0,0.5ex)--node[above]{$\sim$}(0.75,0.5ex);\ \Cat_{\infty/S^{\op}}^{\cocart}\ \tikz[baseline]\draw[>=stealth,->,font=\scriptsize,inner sep=0.5pt](0,0.5ex)--node[above]{$\sim$}(0.75,0.5ex);\ \Fun(S^{\op},\Cat_{\infty}),
\end{equation*}
where the first equivalence is given by unstraightening, the second is given by the formation of the dual, and the last is given by straightening. It is easy to see that all of these equivalences are natural in $S$ \cite[Pr. 3.2.1.4(3)]{HTT}, so we obtain an autoequivalence $\eta$ of the functor $\Fun((-)^{\op},\Cat_{\infty}):\fromto{\Cat_{\infty}^{\op}}{\widehat{\Cat}_{\infty}}$, and thus of the functor 
\begin{equation*}
\Map((-)^{\op},\Cat_{\infty}):\fromto{\Cat_{\infty}^{\op}}{\widehat{\Top}}.
\end{equation*}
Now the left Kan extension of this functor along the inclusion $\into{\Cat_{\infty}^{\op}}{\widehat{\Cat}_{\infty}^{\op}}$ is the functor $h:\fromto{\widehat{\Cat}_{\infty}^{\op}}{\widehat{\Top}}$ represented by $\Cat_{\infty}$. The autoequivalence $\eta$ therefore also extends to an autoequivalence $\widehat{\eta}$ of $h$.

The Yoneda lemma now implies that $\widehat{\eta}$ is induced by an autoequivalence of $\widehat{\Cat}_{\infty}$ itself. By the Unicity Theorem of To\"en \cite{Toen}, Lurie \cite[Th. 4.4.1]{G}, and the first author and Chris Schommer-Pries \cite{BSP}, we deduce that $\widehat{\eta}$ is canonically equivalent either to $\id$ or to $\op$, and considering the case $S=\Delta^0$ shows that it's the former option.

This proves the commutativity of the triangle of equivalences
\begin{equation*}
\begin{tikzpicture} 
\matrix(m)[matrix of math nodes, 
row sep=6ex, column sep=2ex, 
text height=1.5ex, text depth=0.25ex] 
{\Cat_{\infty/S}^{\cart} &&\Cat_{\infty/S^{\op}}^{\cocart}\\ 
&\Fun(S^{\op},\Cat_{\infty}),&\\}; 
\path[>=stealth,<->,inner sep=1pt,font=\scriptsize] 
(m-1-1) edge node[above]{$(-)^{\vee}$} (m-1-3) 
edge node[below left]{$s$} (m-2-2)
(m-1-3) edge node[below right]{$s$} (m-2-2);
\end{tikzpicture}
\end{equation*}
and the commutativity of the remainder of the diagram in Th. \ref{thm:main} follows from duality.
\end{proof}

We've delayed the inevitable long enough. 

\begin{proof}[Proof of Lm. \protect{\ref{prp:doubledual}}] We prove the first assertion; the second is dual. 

To begin, let us unwind the definitions of the duals to describe $X^{\vee\vee}$ explicitly. First, for any $\infty$-category $C$, denote by $\widetilde{\mathcal{O}}^{(2)}(C)$ the simplicial set given by the formula
\begin{equation*}
\widetilde{\mathcal{O}}^{(2)}(C)_k=\Mor((\Delta^k)^{\op}\star\Delta^k\star(\Delta^k)^{\op}\star\Delta^k,C)\cong C_{4k+3}.
\end{equation*}
(This is a two-fold edgewise subdivision of $C$. It can equally well be described as a ``twisted $3$-simplex $\infty$-category of $C$.'') Now the $n$ simplices of $X^{\vee\vee}$ are those functors
\begin{equation*}
x:\fromto{\widetilde{\mathcal{O}}^{(2)}(\Delta^n)^{\op}}{X}
\end{equation*}
such that any $r$-simplex of the form
\begin{equation*}
x(ab_1c_1d_1)\to x(ab_2c_2d_2)\to\cdots\to x(ab_rc_rd_r)
\end{equation*}
covers a totally degenerate $r$-simplex of $S$, and, for any integers
\begin{equation*}
0\leq a'\leq a\leq b\leq b'\leq c'\leq c\leq d\leq d'\leq n
\end{equation*}
(which together represent an edge $\fromto{abcd}{a'b'c'd'}$ of $\widetilde{\mathcal{O}}^{(2)}(C)$) we have
\begin{enumerate}[(\ref{prp:doubledual}.1)]
\item the morphism $\fromto{x(a'bcd)}{x(abcd)}$ is $p$-cartesian;
\item the morphism $\fromto{x(ab'cd)}{x(abcd)}$ is an equivalence;
\item the morphism $\fromto{x(abcd')}{x(abcd)}$ is an equivalence.
\end{enumerate}
In other words, an object of $X^{\vee\vee}$ is an object of $X$, and a morphism of $X^{\vee\vee}$ is a diagram
\begin{equation*}
\begin{tikzpicture} 
\matrix(m)[matrix of math nodes, 
row sep=3ex, column sep=3ex, 
text height=1.5ex, text depth=0.25ex] 
{&u&&v&\\ 
x&&y&&z\\}; 
\path[>=stealth,->,inner sep=0.9pt,font=\scriptsize] 
(m-1-2) edge node[above left]{$\phi$} (m-2-1) 
edge node[above right]{$g$} (m-2-3)
(m-1-4) edge node[above left]{$\psi$} (m-2-3) 
edge node[above right]{$f$} (m-2-5); 
\end{tikzpicture}
\end{equation*}
in $X$ such that $\phi$, $g$, and $\psi$ all cover degenerate edges of $S$, and
\begin{enumerate}[{(\ref{prp:doubledual}}.1{-bis)}]
\item the morphism $f$ is $p$-cartesian;
\item the morphism $\psi$ is an equivalence;
\item the morphism $\phi$ is an equivalence.
\end{enumerate}

We will now construct a cartesian fibration $p':\fromto{X'}{S}$, a trivial fibration $\alpha:\equivto{X'}{X}$ over $S$ and a fiberwise equivalence $\beta:\equivto{X'}{X^{\vee\vee}}$ over $S$. These equivalences will all be the identity on objects. We will identify $X'$ with the subcategory of $X^{\vee\vee}$ whose morphisms are as above with $\psi$ and $\phi$ are degenerate; the inclusion will be the fiberwise equivalence $\beta$. The equivalence $\alpha:\equivto{X'}{X}$ will then in effect be obtained by composing $g$ and $f$.

To construct $p'$, we write, for any $\infty$-category $C$,
\begin{equation*}
\mathcal{O}(C)\coloneq\Fun(\Delta^1,C).
\end{equation*}
Note that the functor $s:\fromto{\mathcal{O}(C)}{C}$ given by evaluation at $0$ is a cartesian fibration (Ex. \ref{exm:scocartfib}). We now define $X'$ as the simplicial set whose $n$-simplices are those commutative squares
\begin{equation*}
\begin{tikzpicture} 
\matrix(m)[matrix of math nodes, 
row sep=4ex, column sep=4ex, 
text height=1.5ex, text depth=0.25ex] 
{\mathcal{O}(\Delta^n)&X\\ 
\Delta^n&S,\\}; 
\path[>=stealth,->,font=\scriptsize] 
(m-1-1) edge node[above]{$x$} (m-1-2) 
edge node[left]{$s$} (m-2-1) 
(m-1-2) edge node[right]{$p$} (m-2-2) 
(m-2-1) edge node[below]{$\sigma$} (m-2-2); 
\end{tikzpicture}
\end{equation*}
such that $x$ carries $s$-cartesian edges to $p$-cartesian edges. We define $p':\fromto{X'}{S}$ to be the map that carries an $n$-simplex as above to $\sigma\in S_n$.

We now construct the desired equivalences. The basic observation is that for any integer $k\geq0$, we have functors
\begin{equation*}
\Delta^k\ot \Delta^k\times\Delta^1\to \Delta^k\star\Delta^k\ \tikz[baseline]\draw[>=stealth,right hook->](0,0.5ex)--(0.5,0.5ex);\ \Delta^k\star(\Delta^k)^{\op}\star\Delta^k\star(\Delta^k)^{\op}:
\end{equation*}
on the left we have the projection onto the first factor; in the middle we have the functor corresponding to the unique natural transformation between the two inclusions $\into{\Delta^k}{\Delta^k\star\Delta^k}$; on the right we have the obvious inclusion. These functors induce, for any $n\geq 0$, functors
\begin{equation*}
\Delta^n\to\mathcal{O}(\Delta^n)\ot \widetilde{\mathcal{O}}^{(2)}(\Delta^n)^{\op}.
\end{equation*}
These in turn induce a zigzag of functors
\begin{equation*}
X\ \tikz[baseline]\draw[>=stealth,<-,font=\scriptsize](0,0.5ex)--node[above]{$\alpha$}(0.5,0.5ex);\ X'\ \tikz[baseline]\draw[>=stealth,->,font=\scriptsize](0,0.5ex)--node[above]{$\beta$}(0.5,0.5ex);\ X^{\vee\vee}
\end{equation*}
over $S$, which are each the identity on objects. On morphisms, $\alpha$ carries a morphism given by $x\to y\to z$ to the composite $\fromto{x}{z}$, and $\beta$ carries a morphism given by $x\to y\to z$ to the morphism of $X^{\vee\vee}$ given by the diagram
\begin{equation*}
\begin{tikzpicture} 
\matrix(m)[matrix of math nodes, 
row sep=3ex, column sep=3ex, 
text height=1.5ex, text depth=0.25ex] 
{&x&&y&\\ 
x&&y&&z.\\}; 
\path[>=stealth,->,inner sep=0.9pt,font=\scriptsize] 
(m-1-2) edge[-,double distance=1.5pt] (m-2-1) 
edge node[above right]{$g$} (m-2-3)
(m-1-4) edge[-,double distance=1.5pt] (m-2-3) 
edge node[above right]{$f$} (m-2-5); 
\end{tikzpicture}
\end{equation*}
We now have the following, whose proof we postpone for a moment.
\begin{lem} \label{olhum}
The morphism $X' \to X$ constructed above is a trivial Kan fibration. Thus $p'$ is the composite of two cartesian fibrations, and therefore a cartesian fibration.
\end{lem}
\noindent Now to complete the proof of Pr. \ref{prp:doubledual}, it suffices to remark that $\fromto{X'}{X^{\vee\vee}}$ is manifestly a fiberwise equivalence.
\end{proof}

Let's now set about proving that $\fromto{X'}{X}$ is indeed a trivial fibration. For this, we will need to make systematic use of the cartesian model categories of marked simplicial sets as presented in \cite[\S 3.1]{HTT}.
\begin{proof}[Proof of Lm. \protect{\ref{olhum}}] We make $\mc{O}(\D^n)$ into a marked simplicial set $\mc{O}(\D^n)^\natural$ by marking those edges that map to degenerate edges under the target map $t : \mc{O}(\D^n) \to \D^n$. Furthermore, for any simplicial subset $K\subset\mc{O}(\D^n)$, let us write $K^{\natural}$ for the marked simplicial set $(K,E)$ in which $E\subset K_1$ is the set of edges that are marked as edges of $\mc{O}(\D^n)^\natural$.

Now write
\begin{equation*}
\del \mc{O}(\D^n)\coloneq\bigcup_{i=0}^n\mc{O}(\D^{\{0,\dots,\hat{i},\dots,n\}})\subset\mc{O}(\D^n),
\end{equation*}
which is a proper simplicial subset of $\Fun(\D^1,\del\D^n)$ when $n > 2$. Observe that $\del \mc{O}(\D^n)$ has the property that there is a bijection
\begin{equation*}
\Map(\del \mc{O}(\D^n),X)\cong\Map(\del \D^n,X').
\end{equation*}

Recasting the statement the Lemma in terms of lifting properties, we see that it will follow from the claim that for any $n \geq 0$ and any morphism $\fromto{\mc{O}(\D^n)^\natural}{S^{\sharp}}$ of marked simplicial sets, the natural inclusion
\[\iota_n: \del \mc{O}(\D^n)^\natural \cup^{(\del \D^n)^\flat} (\D^n)^\flat \to \mc{O}(\D^n)^\natural\]
is a trivial cofibration in the cartesian model structure for marked simplicial sets over $S$, where the $\del \D^n$ in $\del \mc{O}(\D^n)$ is the boundary of the ``long $n$-simplex'' whose vertices are the identity edges in $\D^n$.

In fact, we will prove slightly more. Let $\mathcal{I}$ denote the smallest class of monomorphisms of marked simplicial sets that contains the marked anodyne morphisms and satisfies the two-out-of-three axiom. We call these morphisms \emph{effectively anodyne} maps of marked simplicial sets. Clearly, for any morphism $\fromto{Y}{S^{\sharp}}$, an effectively anodyne morphism $\fromto{X}{Y}$ is a trivial cofibration in the cartesian model structure on marked simplicial sets over $S$.

It's clear that $\iota_1$ is marked anodyne, because it's isomorphic to the inclusion 
\[\into{(\Delta^{\{0, 2\}})^\flat}{(\Delta^2)^\flat \cup^{(\D^{\{1, 2\}})^\flat} (\D^{\{1, 2\}})^\sharp}.\]
Our claim for $n>1$ will in turn follow from the following sublemma.
\begin{lem} \label{ojuny}
The inclusion $\into{(\D^n)^\flat}{\mc{O}(\D^n)^\natural}$ of the ``long $n$-simplex'' is effectively anodyne.
\end{lem}
\noindent Let's assume the veracity of this lemma for now, and let's complete the proof of Lm. \ref{olhum}. It's enough to show that the inclusion
\[\into{(\D^n)^\flat}{\del \mc{O}(\D^n)^\natural \cup^{(\del \D^n)^\flat} (\D^n)^\flat}\]
is effectively anodyne, for then $\iota_n$ will be a effectively anodyne by the two-out-of three property. We'll deploy induction and assume that Lemma \ref{olhum} has been proven for each $l < n$. 
 Now for each $l$, let
\[\widetilde{\text{sk}}_l \mc{O} (\D^n)^\natural := \colim_{I \subseteq n, |I| \leq l} \mc{O}(\D^I)^\natural \]
so that
\[\tsk_{n - 1} \mc{O}(\D^n)^\natural = \del \mc{O}(\D^n)^\natural.\]
By Lemma \ref{olhum} for $\iota_l$, we have that
\[\tsk_{l - 1} \mc{O}(\D^n)^\natural \cup^{(\sk_{l -1} \D^n)^\flat} (\D^n)^\flat \to \tsk_l \mc{O}(\D^n)^\natural \cup^{(\sk_l \D^n)^\flat} (\D^n)^\flat\]
is a trivial cofibration, because it's a composition of pushouts along maps isomorphic to $\iota_l$. Since
\[\tsk_0 \mc{O}(\D^n)^\natural \cup^{(\sk_0 \D^n)^\flat} (\D^n)^\flat = (\D^n)^\flat,\]
 iterating this up to $l = n -1$ gives the result.
\end{proof}

\begin{proof}[Proof of Lm. \protect{\ref{ojuny}}] %We'll call an edge $e$ of $\mc{O}(\D^n)^\natural$ \emph{horizontal} if $s(e)$ is degenerate and \emph{vertical} if $t(e)$ is degenerate, so that an edge is vertical if and only if it's marked.
Write $S$ for the set of nondegenerate $(2n)$-simplices
\[x=\left[00=i_0j_0\to i_1j_1\to\cdots\to i_{2n}j_{2n}=nn\right]\]
of $\mc{O}(\D^n)$. For $x \in S$ as above, define
\[A(x) = \frac 1 2 \left(-n+\sum_{r=0}^{2n} (j_r-i_r)\right).\]
Drawing $\mc{O}(\D^n)$ as a staircase-like diagram and $x$ as a path therein, it's easily checked that $A(x)$ is the number of squares enclosed between $x$ and the ``stairs'' given by the simplex
\[x_0 = \left[00 \to 01 \to 11 \to 12 \to \cdots \to (n-1)n \to nn\right].\]
We'll fill in the simplices of $S$ by induction on $A(x)$. For $k\geq 0$, let 
\[S_k = \{x \in S \, | \, A(x) = k\}\textrm{\quad and\quad}T_k = \{x \in S \, | \, A(x) \leq k\}\]
and
\[\mc{O}_k (\D^n) := \bigcup_{x\in T_k} x\subset\mc{O}(\D^n).\]
We make the convention that
\[\mc{O}_{-1}(\D^n) := \D^n.\]
We must now show that for all $k$ with $0 \leq k \leq \frac 1 2 n (n -1)$, the inclusion
\begin{equation*}
\into{\mc{O}_{k - 1}(\D^n)^\natural}{\mc{O}_k(\D^n)^\natural}
\end{equation*}
is marked anodyne, and for each $k$ it will be a matter of determining $x \cap \mc{O}_{k - 1}(\D^n)$ for each $x \in S_k$ and showing that the inclusion
\[\into{x^{\natural} \cap \mc{O}_{k - 1}(\D^n)^\natural}{x^{\natural}}\]
is effectively anodyne.

The case $k = 0$ is exceptional, so let's do it first. The set $S_0$ has only one element, the simplex
\[x_0 = \left[00 \to 01 \to 11 \to 12 \to \cdots \to (n-1)n \to nn\right].\]
We claim that the inclusion of $\into{\mc{O}_{-1}(\D^n)^{\natural}}{x_0^{\natural}}$ is effectively anodyne. Sticking all the marked 2-simplices of the form
\[\left[ii \to i(i+1) \to (i + 1)(i + 1)\right]^{\natural}\]
onto $\mc{O}_{-1}(\D^n)^{\natural}$ is a marked anodyne operation, so let's do that and call the result $y$. Clearly the spine of $x_0$ is inner anodyne in $y$, so the inclusion $\into{y}{x_0}$ is a trivial cofibration. This proves the claim.

Now we suppose $k>0$, and suppose
\[x=\left[00=i_0j_0\to i_1j_1\to\cdots\to i_{2n}j_{2n}=nn\right]\in S_k.\]
We call a vertex $v=(i_rj_r)$ of $x$ a \emph{flipvertex} if it satisfies the following conditions:
\begin{itemize}
\item $0<r<2n$;
\item $j_r>i_r$;
\item $i_{r-1}=i_r$ (and hence $j_{r-1}=j_r-1$);
\item $j_{r+1}=j_{r}$ (and hence $i_{r+1}=i_r+1$).
\end{itemize}
Observe that $x$ must contain some flipvertices, and it is uniquely determined by them. Note also that if $y$ is an arbitary simplex of $\mc{O}(\D^n)$ containing all the flipvertices of $x$, and if $z \in S$ contains $y$ as a subsimplex, then $A(z) \geq A(x)$, with equality if and only if $z = x$.

We define the \emph{flip of $x$ at $v$} $\Phi(x, v)$ as the modification of $x$ in which the sequence
\[\cdots\to i_r(j_r - 1) \to i_rj_r \to (i_r+1)j_r\to\cdots\]
has been replaced by the sequence
\[\cdots\to i_r(j_r-1) \to (i_r + 1)(j_r - 1) \to (i_r + 1)j_r\to\cdots.\]
Then $\Phi(x, v) \in S_{k - 1}$, so we have $\Phi(x, v) \subset \mc{O}_{k - 1}(\D^n).$ We have therefore established that $x \cap \mc{O}_{k - 1}(\D^n)$ is the union of the faces
\[\del_v x = x \cap \Phi(x, v)\]
as $v$ ranges over flipvertices of $x$. Equivalently, if $\{v_1, \cdots, v_m\}$ is the set of flipvertices of $x$, then $x \cap \mc{O}_{k - 1}(\D^n)$ is the generalized horn
\[x \cap \mc{O}_{k - 1}(\D^n)\cong\Lambda^{2n}_{\{0, \cdots, 2n\} \setminus \{v_1, \cdots, v_m\}}\subset\Delta^{2n}\cong x\]
in the sense of \cite[Nt. 12.6]{mack1}.

If $m > 1$, since flipvertices cannot be adjacent, it follows that the set
\[\{0, \cdots, 2n\} \setminus \{v_1, \cdots, v_m\}\]
satisfies the hypothesis of \cite[Lm. 12.13]{mack1}, and so the inclusion $\into{x \cap \mc{O}_{k -1}(\D^n)}{x}$ is inner anodyne, whence $\into{x^{\natural} \cap \mc{O}_{k -1}(\D^n)^\natural}{x^\natural}$ is effectively anodyne.

On the other hand, if $m=1$, then $x \cap \mc{O}_{k - 1}(\D^n)$ is a face:
\[x \cap \mc{O}_{k - 1}(\D^n)=\del_vx\cong\Delta^{\{0, \dots, \widehat{i+j},\dots, 2n\}}\subset\Delta^{2n}\cong x,\]
where $v=(ij)$ is the unique flipvertex of $x$. We must show that the inclusion
\[\into{x^{\natural} \cap \mc{O}_{k - 1}(\D^n)^{\natural}}{x^{\natural}}\]
is effectively anodyne. We denote by $y$ the union of $\del_v x$ with the 2-simplex
\[[ i(j - 1) \to ij \to (i+1)j].\]
The inclusion $\into{\del_v x^{\natural}}{y^{\natural}}$ is marked anodyne; we claim that the inclusion $\into{y}{x}$ is inner anodyne.

Indeed, something more general is true: suppose $s$ is an inner vertex of $\D^m$ and $F$ is a subset of $[m]$ which has $s$ as an inner vertex and is \emph{contiguous}, meaning that if $t_1, t_2 \in F$ and $t_1 < u < t_2$ then $u \in F$. Then the inclusion $\del_s \D^m \cup \D^F \to \D^m$ is inner anodyne.

We prove this by induction on $m - |F|$. If $|F| = m$, then $\D^F = \D^m$ and the claim is vacuous. Otherwise, let $F'$ be a contiguous subset of $[n]$ containing $F$ with $|F'| = |F| + 1$. Then 
\[\D^{F'} \cap (\D^F \cup \del_s \D^m) = \D^F \cup \del_s \D^{F'}.\]
But $\D^F \cup \del_s \D^{F'}$ is the generalized horn $\Lambda^{F'}_{F \setminus \{s\}}$, and $F \setminus \{s\}$ satisfies the hypothesis of \cite[Lm. 12.13]{mack1} as a subset of $F'$ since $s$ was already an inner vertex of $F$. Thus $\del_s \D^n \cup \D^F \to \del_s \D^n \cup \D^{F'}$ is inner anodyne, and by the induction hypothesis, we are done.
\end{proof}

%----------------------------------------------------------------------%

\section{The duality pairing}\label{reltwarr} In this section we give construct a pairing that concretely exhibits the equivalence between the functor $\YY:\fromto{T}{\Cat_{\infty}}$ that classifies a cocartesian fibration $q:\fromto{Y}{T}$ and the opposite of the functor that classifies the cocartesian fibration $(q^{\vee})^{\op}$.

The way we'll go about this is the following: we will construct a left fibration
\begin{equation*}
M:\fromto{\widetilde{\mathcal{O}}(Y/T)}{(Y^{\vee})^{\op}\times_T Y}
\end{equation*}
such that for any object $t\in T$, the pulled back fibration
\begin{equation*}
\fromto{\widetilde{\mathcal{O}}(Y/T)_t}{((Y^{\vee})^{\op})_t\times Y_t\simeq Y_t^{\op}\times Y_t}
\end{equation*}
is a \textbf{\emph{perfect pairing}}; i.e., it satisfies the conditions of the following result of Lurie.

\begin{prp}[\protect{\cite[Cor. 4.2.14]{DAGX}}]\label{lem:twarrrecog} Suppose $\sigma:\fromto{X}{A}$ and $\tau:\fromto{X}{B}$ two functors that together define a left fibration $\lambda:\fromto{X}{A\times B}$. Then $\lambda$ is equivalent to a fibration of the form $\fromto{\widetilde{\mathcal{O}}(C)}{C^{\op}\times C}$ (and in particular $A\simeq B^{\op}$) just in case the following conditions are satisfied.
\begin{enumerate}[(\ref{lem:twarrrecog}.1)]
\item For any object $a\in A$, there exists an initial object in the $\infty$-category $X_a\coloneq \sigma^{-1}(\{a\})$.
\item For any object $b\in B$, there exists an initial object in the $\infty$-category $X_b\coloneq \tau^{-1}(\{b\})$.
\item An object $x\in X$ is initial in $X_{\sigma(x)}$ just in case it is initial in $X_{\tau(x)}$.
\end{enumerate}
\end{prp}

In our case, the functor that classifies $M$ will be the \textbf{\emph{fiberwise mapping space}} functor
\begin{equation*}
\Map_{Y/T}:\fromto{(Y^{\vee})^{\op}\times_T Y}{\Top}.
\end{equation*}
This functor carries an object $(x,y)\in (Y^{\vee})^{\op}\times_T Y$ to the space $\Map_{\YY(t)}(x,y)$, where $t=q(x)=q(y)$. If $\phi:\fromto{s}{t}$ is a morphism of $S$, then a morphism
\begin{equation*}
(f,g):\fromto{(u,v)}{(x,y)}
\end{equation*}
of $(Y^{\vee})^{\op}\times_T Y$ covering $\phi$ is given, in effect, by morphisms $f:\fromto{x}{\YY(\phi)(u)}$ and $g:\fromto{\YY(\phi)(v)}{y}$ of $\YY(s)$. The functor $\Map_{Y/T}$ will then carry $(f,g)$ to the morphism
\begin{equation*}
\Map_{\YY(s)}(u,v)\ \tikz[baseline]\draw[>=stealth,->,font=\scriptsize](0,0.5ex)--node[above]{$\YY(\phi)$}(0.75,0.5ex);\ \Map_{\YY(t)}(\YY(\phi)(u),\YY(\phi)(v))\ \tikz[baseline]\draw[>=stealth,->,font=\scriptsize](0,0.5ex)--node[above]{$g\circ-\circ f$}(1.25,0.5ex);\ \Map_{\YY(t)}(x,y).
\end{equation*}

\begin{nul}\label{nul:OtildeoverT} Before we proceed headlong into the details of the construction, let us first give an informal but very concrete description of both $\widetilde{\mathcal{O}}(Y/T)$ and $M$. The objects of $\widetilde{\mathcal{O}}(Y/T)$ will be morphisms $f:\fromto{u}{v}$ of $Y$ such that $q(f)$ is an identity morphism in $T$. Now a morphism $\fromto{f}{g}$ from an arrow $f:\fromto{u}{v}$ to an arrow $g:\fromto{x}{y}$ is a commutative diagram
\begin{equation*}
\begin{tikzpicture}[baseline]
\matrix(m)[matrix of math nodes, 
row sep=3ex, column sep=4ex, 
text height=1.5ex, text depth=0.25ex] 
{u&&x\\[-4ex]
&w&\\ 
v&&y\\}; 
\path[>=stealth,->,font=\scriptsize] 
(m-1-3) edge node[above]{$\psi$} (m-2-2)
edge node[right]{$g$} (m-3-3)
(m-1-1) edge node[above]{$\phi$} (m-2-2) 
edge node[left]{$f$} (m-3-1) 
(m-2-2) edge (m-3-3) 
(m-3-1) edge node[below]{$\xi$} (m-3-3); 
\end{tikzpicture}
\end{equation*}
in which $\phi$ is $q$-cocartesian, $q(\psi)$ is an identity morphism. Composition is performed by forming suitable pushouts on the source side and simple composition on the target side. We will establish below that there is indeed an $\infty$-category that admits this description.

The functor $M$ will carry an object $f\in\widetilde{\mathcal{O}}(Y/T)$ as above to the pair of objects $(u,v)\in(Y^{\vee})^{\op}\times Y$, and it will carry a morphism $\fromto{f}{g}$ as above to the pair of morphisms
\begin{equation*}
\left(\begin{tikzpicture}[baseline]
\matrix(m)[matrix of math nodes, 
row sep=2ex, column sep=3ex, 
text height=1.5ex, text depth=0.25ex] 
{&w&\\ 
u&&x\\}; 
\path[>=stealth,<-,inner sep=0.9pt,font=\scriptsize] 
(m-1-2) edge node[above left]{$\phi$} (m-2-1) 
edge node[above right]{$\psi$} (m-2-3); 
\end{tikzpicture}
,\ v\ \tikz[baseline]\draw[>=stealth,->,font=\scriptsize](0,0.5ex)--node[above]{$\xi$}(0.5,0.5ex);\ y\right)\in(Y^{\vee})^{\op}\times Y.
\end{equation*}

We call $M$ the \textbf{\emph{duality pairing}} for $q$. We will prove below that it is left fibration, whence it follows readily from this description that the functor that classifies it is indeed be the fiberwise mapping space functor
\begin{equation*}
\Map_{Y/T}:\fromto{(Y^{\vee})^{\op}\times_T Y}{\Top}
\end{equation*}
defined above.
\end{nul}

\begin{prp}\label{prp:Otildeconstruct} Both an $\infty$-category $\widetilde{\mathcal{O}}(Y/T)$ and a left fibration $M$ as described above exist.
\end{prp}

We postpone the precise construction of $\widetilde{\mathcal{O}}(Y/T)$ and $M$ till the end of this section (Constr. \ref{constr:buildotilde}). Our concrete description suffices to deduce the main result of this section.
\begin{thm} For any object $t\in T$, the left fibration
\begin{equation*}
\fromto{\widetilde{\mathcal{O}}(Y/T)_t}{((Y^{\vee})^{\op})_t\times Y_t}
\end{equation*}
pulled back from the duality pairing $M$ is a perfect pairing; i.e., it satisfies the conditions of Pr. \ref{lem:twarrrecog}.
\begin{proof} Suppose $x\in((Y^{\vee})^{\op})_t$ and $y\in Y_t$. Then it is easy to see that the identity map $\id_x$ is the initial object of the fiber $\widetilde{\mathcal{O}}(Y/T)_{x}$: for any morphism $g:\fromto{x}{y}$ such that $q(g)$ is a degenerate edge, the essentially unique morphism $\fromto{\id_x}{g}$ of $\widetilde{\mathcal{O}}(Y/T)_{x}$ is given by the diagram
\begin{equation*}
\begin{tikzpicture}[baseline]
\matrix(m)[matrix of math nodes, 
row sep=3ex, column sep=4ex, 
text height=1.5ex, text depth=0.25ex] 
{x&&x\\[-4ex]
&x&\\ 
x&&y\\}; 
\path[>=stealth,->,font=\scriptsize] 
(m-1-3) edge[-,double distance=1.5pt] (m-2-2)
edge node[right]{$g$} (m-3-3)
(m-1-1) edge[-,double distance=1.5pt] (m-2-2) 
edge[-,double distance=1.5pt] (m-3-1) 
(m-2-2) edge[inner sep=0.95pt] node[below left]{$g$} (m-3-3) 
(m-3-1) edge node[below]{$g$} (m-3-3); 
\end{tikzpicture}
\end{equation*}
Dually, the identity map $\id_y$ is the initial object of the fiber $\widetilde{\mathcal{O}}(Y/T)_{y}$: the essentially unique morphism $\fromto{\id_y}{g}$ of $\widetilde{\mathcal{O}}(Y/T)_{y}$ is given by the diagram
\begin{equation*}
\begin{tikzpicture}[baseline]
\matrix(m)[matrix of math nodes, 
row sep=3ex, column sep=4ex, 
text height=1.5ex, text depth=0.25ex] 
{y&&x\\[-4ex]
&y&\\ 
y&&y\\}; 
\path[>=stealth,->,font=\scriptsize] 
(m-1-3) edgenode[above]{$g$} (m-2-2)
edge node[right]{$g$} (m-3-3)
(m-1-1) edge[-,double distance=1.5pt] (m-2-2) 
edge[-,double distance=1.5pt] (m-3-1) 
(m-2-2) edge[-,double distance=1.5pt] (m-3-3) 
(m-3-1) edge[-,double distance=1.5pt] (m-3-3); 
\end{tikzpicture}
\end{equation*}
The result now follows immediately.
\end{proof}
\end{thm}

In light of Pr. \ref{lem:twarrrecog}, we deduce an identification
\begin{equation*}
((Y^{\vee})^{\op})_t\simeq Y_t^{\op}
\end{equation*}
that is functorial in $t$, as desired.

\begin{cnstr}\label{constr:buildotilde} We now set about giving a precise construction of the $\infty$-category $\widetilde{\mathcal{O}}(Y/T)$ and the left fibration $M$ described in \ref{nul:OtildeoverT}. We use very heavily the technology of effective Burnside $\infty$-categories from \cite{mack1}.

We begin by identifying two subcategories of the arrow $\infty$-category $\mathcal{O}(Y)$, each of which contains all the objects. Suppose $f:\fromto{u}{v}$ and $g:\fromto{x}{y}$ morphisms of $Y$. A morphism $\eta:\fromto{f}{g}$ of $\mathcal{O}(Y)$ given by a square
\begin{equation*}
\begin{tikzpicture} 
\matrix(m)[matrix of math nodes, 
row sep=4ex, column sep=4ex, 
text height=1.5ex, text depth=0.25ex] 
{u&v\\ 
x&y\\}; 
\path[>=stealth,->,font=\scriptsize] 
(m-1-1) edge node[above]{$f$} (m-1-2) 
edge node[left]{$s(\eta)$} (m-2-1) 
(m-1-2) edge node[right]{$t(\eta)$} (m-2-2) 
(m-2-1) edge node[below]{$g$} (m-2-2); 
\end{tikzpicture}
\end{equation*}
lies in $\mathcal{O}(Y)_{\dag}$ just in case $q(s(\eta))$ is an equivalence of $T$ and $t(\eta)$ is an equivalence of $Y$; the morphism $\eta$ lies in $\mathcal{O}(Y)^{\dag}$ just in case $s(\eta)$ is $q$-cocartesian.

Now form the effective Burnside $\infty$-categories
\begin{eqnarray*}
\widehat{\widetilde{\mathcal{O}}'(Y)}&\coloneq&A^{\eff}(\mathcal{O}(Y)^{\op},(\mathcal{O}(Y)_{\dag})^{\op},(\mathcal{O}(Y)^{\dag})^{\op}),\\
\widehat{\mathcal{O}(T)}&\coloneq&A^{\eff}(\mathcal{O}(T)^{\op},\iota\mathcal{O}(T)^{\op},\mathcal{O}(T)^{\op}),\\
\widehat{(Y^{\vee})^{\op}}&\coloneq&A^{\eff}(Y^{\op},Y^{\op}\times_{T^{\op}}\iota T^{\op},(\iota_TY)^{\op}),\\
\widehat{Y}&\coloneq&A^{\eff}(Y^{\op},\iota Y^{\op},Y^{\op}),\\
\widehat{T}&\coloneq&A^{\eff}(T^{\op},\iota T^{\op},T^{\op}).
\end{eqnarray*}
The objects of $\widehat{\widetilde{\mathcal{O}}'(Y)}$ are thus morphisms $f:\fromto{u}{v}$ of $Y$, and a morphism $\fromto{f}{g}$ from an arrow $f:\fromto{u}{v}$ to an arrow $g:\fromto{x}{y}$ is a commutative diagram
\begin{equation*}
\begin{tikzpicture}[baseline]
\matrix(m)[matrix of math nodes, 
row sep=4ex, column sep=4ex, 
text height=1.5ex, text depth=0.25ex] 
{u&u'&x\\
v&y'&y\\}; 
\path[>=stealth,->,font=\scriptsize] 
(m-1-1) edge node[above]{$\phi$} (m-1-2)
edge node[left]{$f$} (m-2-1)
(m-1-2) edge (m-2-2)
(m-2-1) edge node[below]{$\xi$} (m-2-2)
(m-1-3) edge node[above]{$\psi$} (m-1-2)
edge node[right]{$g$} (m-2-3)
(m-2-3) edge node[below]{$\eta$} (m-2-2); 
\end{tikzpicture}
\end{equation*}
in which: $\phi$ is $q$-cocartesian, $q(\psi)$ is an equivalence, and $\eta$ is an equivalence.

The source and target functors $\fromto{\mathcal{O}(Y)^{\op}}{Y^{\op}}$ along with the cocartesian fibration $q$ together induce a diagram of functors
\begin{equation*}
\begin{tikzpicture} 
\matrix(m)[matrix of math nodes, 
row sep=5ex, column sep=4ex, 
text height=1.5ex, text depth=0.25ex] 
{\widehat{\widetilde{\mathcal{O}}'(Y)}&\widehat{\mathcal{O}(T)}\\ 
\widehat{(Y^{\vee})^{\op}}\times\widehat{Y}&\widehat{T}\times\widehat{T}\\}; 
\path[>=stealth,->,font=\scriptsize] 
(m-1-1) edge (m-1-2) 
edge (m-2-1) 
(m-1-2) edge (m-2-2) 
(m-2-1) edge (m-2-2); 
\end{tikzpicture}
\end{equation*}
Observe that the omnibus theorem of the first author \cite[Th. 12.2]{mack1} implies that all of the functors that appear in this quadrilateral are inner fibrations.

Furthermore, since the formation of the effective Burnside $\infty$-category respects fiber products, one may employ \cite[Th. 12.2]{mack1} to show not only that the natural map 
\begin{equation*}
\widehat{M}':\fromto{\widehat{\widetilde{\mathcal{O}}'(Y)}}{\left(\widehat{(Y^{\vee})^{\op}}\times\widehat{Y}\right)\underset{\widehat{T}\times\widehat{T}}{\times}\widehat{\mathcal{O}(T)}}
\end{equation*}
is an inner fibration, but also that every morphism of $\widehat{\widetilde{\mathcal{O}}'(Y)}$ is $\widehat{M}'$-cocartesian. It is clear that $\widehat{M}'$ admits the right lifting property with respect to the inclusion $\into{\Delta^{\{0\}}}{\Delta^1}$, one deduces that $\widehat{M}'$ is a left fibration.

As we see, the $\infty$-category $\widehat{\widetilde{\mathcal{O}}'(Y)}$ is much too large, but we now proceed to cut both it and the left fibration $\widehat{M}'$ down to size via pullbacks:
\begin{enumerate}[(\ref{constr:buildotilde}.1)]
\item The first pullback in effect requires all equivalences in the description of the morphisms of $\widehat{\widetilde{\mathcal{O}}'(Y)}$ above to be identities. We pull back $\widehat{M}'$ along the inclusion
\begin{equation*}
\into{\left((Y^{\vee})^{\op}\times Y\right)\underset{T\times T}{\times}\mathcal{O}(T)}{\left(\widehat{(Y^{\vee})^{\op}}\times\widehat{Y}\right)\underset{\widehat{T}\times\widehat{T}}{\times}\widehat{\mathcal{O}(T)}}
\end{equation*}
(which is of course an equivalence) to obtain a left fibration
\begin{equation*}
M':\fromto{\widetilde{\mathcal{O}}'(Y)}{\left((Y^{\vee})^{\op}\times Y\right)\underset{T\times T}{\times}\mathcal{O}(T)}.
\end{equation*}
\item Second, we pull back the composite
\begin{equation*}
\widetilde{\mathcal{O}}'(Y)\ \tikz[baseline]\draw[>=stealth,->,font=\scriptsize](0,0.5ex)--node[above]{$M'$}(0.5,0.5ex);\ \left((Y^{\vee})^{\op}\times Y\right)\underset{T\times T}{\times}\mathcal{O}(T)\to\mathcal{O}(T)
\end{equation*}
along the inclusion $\into{T}{\mathcal{O}(T)}$ of the degenerate arrows to obtain the desired left fibration
\begin{equation*}
M:\fromto{\widetilde{\mathcal{O}}(Y/T)}{(Y^{\vee})^{\op}\times_T Y}
\end{equation*}
\end{enumerate}

It is now plain to see that $\widetilde{\mathcal{O}}(Y/T)$ is the $\infty$-category described in \ref{nul:OtildeoverT}, and $M$ is the left fibration described there.
\end{cnstr}

%----------------------------------------------------------------------%
%----------------------------------------------------------------------%

\appendix

\section{Cartesian and cocartesian fibrations}\label{sect:recall}

\begin{dfn}\label{rec:cocart} Suppose $p\colon\fromto{X}{S}$ an inner fibration of simplicial sets. Recall \cite[Rk. 2.4.1.4]{HTT} that an edge $f\colon\fromto{\Delta^{1}}{X}$ is \textbf{\emph{$p$-cartesian}} just in case, for each integer $n\geq 2$, any extension
\begin{equation*}
\begin{tikzpicture} 
\matrix(m)[matrix of math nodes, 
row sep=4ex, column sep=4ex, 
text height=1.5ex, text depth=0.25ex] 
{\Delta^{\{n-1,n\}}&X,\\ 
\Lambda^n_n&\\}; 
\path[>=stealth,->,font=\scriptsize] 
(m-1-1) edge node[above]{$f$} (m-1-2) 
edge[right hook->] (m-2-1) 
(m-2-1) edge node[below]{$F$} (m-1-2); 
\end{tikzpicture}
\end{equation*}
and any solid arrow commutative diagram
\begin{equation*}
\begin{tikzpicture} 
\matrix(m)[matrix of math nodes, 
row sep=4ex, column sep=4ex, 
text height=1.5ex, text depth=0.25ex] 
{\Lambda^n_n&X\\ 
\Delta^n&S,\\}; 
\path[>=stealth,->,font=\scriptsize] 
(m-1-1) edge node[above]{$F$} (m-1-2) 
edge[right hook->] (m-2-1) 
(m-1-2) edge node[right]{$p$} (m-2-2) 
(m-2-1) edge (m-2-2)
(m-2-1) edge[dotted] node[below]{$\overline{F}$} (m-1-2);
\end{tikzpicture}
\end{equation*}
the dotted arrow $\overline{F}$ exists, rendering the diagram commutative.

We say that $p$ is a \textbf{\emph{cartesian fibration}} \cite[Df. 2.4.2.1]{HTT} if, for any edge $\eta\colon\fromto{s}{t}$ of $S$ and for every vertex $x\in X_0$ such that $p(x)=s$, there exists a $p$-cartesian edge $f\colon\fromto{x}{y}$ such that $\eta=p(f)$.

\textbf{\emph{Cocartesian edges}} and \textbf{\emph{cocartesian fibrations}} are defined dually, so that an edge of $X$ is $p$-cocartesian just in case the corresponding edge of $X^{\op}$ is $p^{\op}$-cartesian, and $p$ is a cocartesian fibration just in case $p^{\op}$ is a cartesian fibration.
\end{dfn}

\begin{exm}\label{exm:nerveofGrothfib} A functor $p\colon\fromto{D}{C}$ between ordinary categories is a Grothendieck fibration if and only if the induced functor $N(p)\colon\fromto{ND}{NC}$ on nerves is a cartesian fibration \cite[Rk 2.4.2.2]{HTT}.
\end{exm}

\begin{exm}\label{exm:scocartfib} For any $\infty$-category $C$, write $\mathcal{O}(C)\coloneq\Fun(\Delta^1,C)$. By \cite[Cor. 2.4.7.12]{HTT}, evaluation at $0$ defines a cartesian fibration $s\colon\fromto{\mathcal{O}(C)}{C}$, and evaluation at $1$ defines a cocartesian fibration $t\colon\fromto{\mathcal{O}(C)}{C}$.

One can ask whether the functor $s\colon\fromto{\mathcal{O}(C)}{C}$ is also a \emph{cocartesian} fibration. One may observe \cite[Lm. 6.1.1.1]{HTT} that an edge $\fromto{\Delta^1}{\mathcal{O}(C)}$ is $s$-cocartesian just in case the corresponding diagram
\begin{equation*}
\fromto{(\Lambda^2_0)^{\rhd}\cong\Delta^1\times\Delta^1}{C}
\end{equation*}
is a pushout square.
\end{exm}

\begin{nul}\label{rec:straighten} Suppose $S$ a simplicial set. Then the collection of cartesian fibrations to $S$ with small fibers is naturally organized into an $\infty$-category $\Cat_{\infty/S}^{\cart}$. To define it, let $\Cat_{\infty}^{\cart}$ be the following subcategory of $\mathcal{O}(\Cat_{\infty})$: an object $\fromto{X}{U}$ of $\mathcal{O}(\Cat_{\infty})$ lies in $\Cat_{\infty}^{\cart}$ if and only if it is a cartesian fibration, and a morphism $\fromto{p}{q}$ in $\mathcal{O}(\Cat_{\infty})$ between cocartesian fibrations represented as a square
\begin{equation*}
\begin{tikzpicture} 
\matrix(m)[matrix of math nodes, 
row sep=4ex, column sep=4ex, 
text height=1.5ex, text depth=0.25ex] 
{X&Y\\ 
U&V\\}; 
\path[>=stealth,->,font=\scriptsize] 
(m-1-1) edge node[above]{$f$} (m-1-2) 
edge node[left]{$p$} (m-2-1) 
(m-1-2) edge node[right]{$q$} (m-2-2) 
(m-2-1) edge (m-2-2); 
\end{tikzpicture}
\end{equation*}
lies in in $\Cat_{\infty}^{\cart}$ if and only if $f$ carries $p$-cartesian edges to $q$-cartesian edges. We now define $\Cat_{\infty/S}^{\cocart}$ as the fiber over $S$ of the target functor
\begin{equation*}
t\colon\Cat_{\infty}^{\cart}\subset\mathcal{O}(\Cat_{\infty})\ {\tikz[baseline]\draw[>=stealth,->](0,0.5ex)--(0.5,0.5ex);}\ \Cat_{\infty}.
\end{equation*}
Equivalently \cite[Pr. 3.1.3.7]{HTT}, one may describe $\Cat_{\infty/S}^{\cart}$ as the simplicial nerve of the (fibrant) simplicial category of marked simplicial sets \cite[Df. 3.1.0.1]{HTT} over $S$ that are fibrant for the \emph{cartesian model structure} --- i.e., of the form $\fromto{X^{\natural}}{S}$ for $\fromto{X}{S}$ a cartesian fibration \cite[Df. 3.1.1.8]{HTT}.

The straightening/unstraightening Quillen equivalence of \cite[Th. 3.2.0.1]{HTT} now yields an equivalence of $\infty$-categories
\begin{equation*}
\Cat_{\infty/S}^{\cart}\simeq\Fun(S^{\op},\Cat_{\infty}).
\end{equation*}
So we obtain a dictionary between cartesian fibrations $p\colon\fromto{X}{S}$ with small fibers and functors $\XX\colon\fromto{S^{\op}}{\Cat_{\infty}}$. For any vertex $s\in S_0$, the value $\XX(s)$ is equivalent to the fiber $X_s$, and for any edge $\eta\colon\fromto{s}{t}$, the functor $\fromto{\XX(t)}{\XX(s)}$ assigns to any object $y\in X_t$ an object $x\in X_s$ with the property that there is a cocartesian edge $\fromto{x}{y}$ that covers $\eta$. We say that $\XX$ \textbf{\emph{classifies}} $p$ \cite[Df. 3.3.2.2]{HTT}.

Dually, the collection of cocartesian fibrations to $S$ with small fibers is naturally organized into an $\infty$-category $\Cat_{\infty/S}^{\cocart}$, and the straightening/unstraightening Quillen equivalence yields an equivalence of $\infty$-categories
\begin{equation*}
\Cat_{\infty/S}^{\cocart}\simeq\Fun(S,\Cat_{\infty}).
\end{equation*}
\end{nul}

\begin{ntn}\label{rec:leftfib} A cartesian (respectively, cocartesian) fibration with the property that each fiber is a Kan complex --- or equivalently, with the property that the functor that classifies it factors through the full subcategory $\Top\subset\Cat_{\infty}$ of Kan complexes --- is called a \textbf{\emph{right}} (resp., \emph{\textbf{left}}) \textbf{\emph{fibration}}. These are more efficiently described as maps that satisfy the right lifting property with respect to horn inclusions $\into{\Lambda^n_k}{\Delta^n}$ such that $1\leq k\leq n$ (resp., $0\leq k\leq n-1$) \cite[Pr. 2.4.2.4]{HTT}.

For any cartesian (resp., cocartesian) fibration $p\colon\fromto{X}{S}$, one may consider the smallest simplicial subset $\iota^SX\subset X$ that contains the $p$-cartesian (resp., $p$-cocartesian) edges. The restriction $\iota^S(p)\colon\fromto{\iota^SX}{S}$ of $p$ to $\iota^SX$ is a right (resp., left) fibration. The functor $\fromto{S^{\op}}{\Top}$ (resp., $\fromto{S}{\Top}$) that classifies $\iota^Sp$ is then the functor given by the composition $\iota\circ\XX$, where $\XX$ is the functor that classifies $p$, and $\iota$ is the functor $\fromto{\Cat_{\infty}}{\Top}$ that extracts the maximal Kan complex contained in an $\infty$-category.
\end{ntn}

%----------------------------------------------------------------------%
%----------------------------------------------------------------------%

\bibliographystyle{amsplain}
\bibliography{cocart}

\end{document}